\documentclass{article}
\usepackage{graphicx} 

 \usepackage{amssymb,amsmath}
 \usepackage{amsthm}
 \usepackage{epic}
 \usepackage{color}
 \usepackage{xargs}
 \usepackage[left=1.0in,top=1in,right=1.0in,bottom=1in]{geometry}

\newtheorem{theorem}{Theorem}
\newtheorem{proposition}[theorem]{Proposition}
\newtheorem{lemma}[theorem]{Lemma}

\newtheorem{claim}{Claim}

\newcommand{\cL}{\mathcal{L}}
 \newcommand{\ovd}[1]{\overrightarrow{#1}}
 \newcommand{\too}{\rightarrowtail}

\begin{document}
\title{Lines on digraphs of low diameter \footnote{
Supported by  Centro de Modelamiento Matemático (CMM), FB210005, BASAL funds for centers of excellence, Fondecyt/Regular 1180994 and Doctoral Fellowship grant 21211955, 
all from ANID-Chile, and {PAPIIT-M\'exico IN108121, CONACyT-M\'exico 282280.}}}
\author{Gabriela Araujo-Pardo$^1$, Martín Matamala$^2$, Juan P. Peña$^2$ and Jos\'e Zamora$^3$\\
\\
  ($1$) Instituto de Matemáticas, Universidad Nacional Aut\'onoma de M\'exico, M\'exico\\
  ($2$) DIM-CMM, CNRS-IRL 2807, Universidad de Chile, Chile\\
  ($3$) Depto. de Matemáticas, Universidad Andres Bello, Chile.}

\maketitle
\begin{abstract} A set of $n$ non-collinear points in the Euclidean plane defines at least $n$ different lines. Chen and Chvátal in 2008 conjectured that the same results is true in metric spaces for an adequate definition of line.

More recently, it was conjectured in 2018 by Aboulker {\it et al.} that any large enough bridgeless graph on $n$ vertices defines a metric space that has at least $n$ lines.

We study the natural extension of Aboulker {\it et al.}'s conjecture into the context of quasi-metric spaces defined by digraphs of low diameter. We prove that it is valid for quasi-metric spaces defined by bipartite digraphs of diameter at most three, oriented graphs of diameter two and, digraphs of diameter three and directed girth four.
\end{abstract}

\section{Introduction}\label{s:intro}

In this work we consider digraphs with diameter two or three in the context of Metric Geometry. We use the quasi-metric defined by the lengths of shortest paths in a digraphs to define the notion of \emph{segment} and \emph{line} in these digraphs as done in \cite{AM20}.
Besides studying properties of these objects, we are mainly interested in counting them in the context of the Chen and Chvátal's conjecture \cite{CC}. This conjecture asks whether every set of non-collinear points in a metric space determines at least as many lines as points the set has. If true, it will generalize to metric spaces a well-known result known in the context of Incidence Geometry as de Bruijn and Erdös's Theorem (see \cite{chvatal_2018} for a more detailed account).

In the last fifteen years several authors have proved that this is indeed the case for some families of graphs and digraphs. On the one hand, in \cite{ABMZ} it was proved that metric spaces defined by \{house,hole\}-free graphs have this property. Previously, this was known for metric spaces defined by chordal graphs \cite{BBCCCCFZ}, distance-hereditary graphs \cite{AK}, graphs such that every induced subgraph is either a chordal graph, has a cut-vertex or a non-trivial module \cite{AMRZ},
bisplit graphs \cite{BKR} and $(q,q-4)$-graphs \cite{SS}. On the other hand, in \cite{AM20} it was proved that the same is true for quasi-metric spaces defined by tournaments and for bipartite tournaments of diameter at most three.

The Chen and Chvátal's conjeture has also been verified for metric spaces with distances $0,1$ or $2$
\cite{ChCh11, Chvatal2} or defined by points in the plane with the Manhattan metric, when no two points lay in a horizontal or in a vertical line  \cite{KP2013}.

In \cite{AMRZ} it was conjectured that there is a finite family of graphs such that each bridgeless graph not in the family defines a metric space with at least as many lines as vertices. Equivalently, metric spaces defined by large enough bridgeless graphs have at least as many lines as vertices. Theorem 5.3 and Theorem 7.5 in \cite{ACHKS} imply that metric spaces with distances in $\{0,1,2,3\}$ and those defined by graphs of bounded diameter, respectively, have this stronger property.
In \cite{MZ2020} this fact was proved for every bridgeless bipartite graph non-isomorphic to the complete bipartite graphs $K_{2,2}$ or $K_{2,3}$.

\subsubsection*{Our results}

Besides the standard notions as given in  \cite{BG2001} and \cite{D2017}, in this work we use the following non-standard definitions and/or assumptions.
\begin{itemize}
 \item All digraphs are assumed to be strongly connected.
 \item A path in a digraph is a directed path.
\item A graph is a digraph $G=(V,A)$ which satisfies, for each pair of distinct vertices $u$ and $v$, that  $(u,v)\in A \iff (v,u)\in A$.
 \item The \emph{directed girth} of a digraph is the length of a shortest directed cycle. Then all digraphs have finite directed girth, a graph has directed girth two and an oriented graph has directed girth at least three.
\item A digraph $G=(V,A)$ admitting a partition\footnote{$X$ and $Y$ are not empties, do not intersect and their union is $V$.} $\{X,Y\}$ of $V$ such that for each $(u,v)\in A$, $|\{u,v\}\cap X|=|\{u,v\}\cap Y|=1$ is called \emph{bipartite}.
\item An arc $(u,v)$ of a digraphs $G$ is a \emph{bridge} if $G-(u,v)$ has no path from $u$ to $v$.
\item A digraph with no bridge is a \emph{bridgeless} digraph.
\end{itemize}

To ease the presentation we say that a digraph is \emph{thin} when it defines a quasi-metric space with less lines than vertices.

From the result in \cite{ACHKS} the class of thin graphs of diameter at most three is finite. Recently, 
a complete characterization of quasimetric spaces with three or four lines was given \cite{GMP2024}, implying that these quasimetric spaces have large diameters. 
Hence, it is reasonable to think that there are only finitely many \emph{thin} quasimetric spaces with low diameter. Whether this is true even for quasimetric spaces defined by digraphs is an open problem. In this work we made modest advances into answering this question. 

It is easy to see that there is no thin digraphs of diameter one and more than two vertices: Digraphs of diameter one are precisely the complete graphs and for each of them,  each edge defines a different line. Hence, none of them is thin, when they have more than two vertices.

Knowing which bridgeless digraphs of diameter two are thin is already a challenge. 
So far, we know ten thin digraphs of diameter two;  All of them are bridgeless \emph{graphs} with the exception of the directed cycle of length three $\overrightarrow{C_3}$.
This fact motivates Theorem \ref{t:oriented2}, where we show that $\overrightarrow{C_3}$ is the only thin oriented graph of diameter two.

Among the nine known thin bridgeless graphs of diameter two, there are only two which are bipartite graphs: the complete bipartite graphs $K_{2,2}$ and $K_{2,3}$. As we mentioned above, it was proved in \cite{MZ2020} that these are the only thin bridgeless bipartite graphs, regardless of their diameters.
We prove in Theorem \ref{t:bip} that this is also true in the class of bridgeless bipartite digraphs of diameter at most three. In fact, we prove that a thin bipartite digraph of diameter at most three either has a bridge, or it is $K_{2,2}$ or $K_{2,3}$.

For thin bridgeless digraphs of diameter three Theorem \ref{t:bip} tells us that none of them is bipartite. So far we know that there are seven thin digraphs of diameter three, three of them are bridgeless graphs and 
four of them are oriented graphs, with bridges and directed girth four.
Our last result is Theorem \ref{t:briddiam3girth4}, where we prove that
a thin digraph of diameter three either has a bridge or has directed girth less than four.

\section{First results and Notation}

Let $G=(V,A)$ be a digraph. The \emph{segment} defined by the pair $(x,y)$ is the set of all vertices $z$ in a shortest path from $x$ to $y$. It is denoted by $[xy]^G$.
Then,
$$[xy]^G=\{z\mid d_G(x,y)=d_G(x,z)+d_G(z,y)\},$$
where for two vertices $x$ and $y$, $d_G(x,y)$ is the length of a shortest path from $x$ to $y$ in $G$. It is worth to stress that in an arbitrary digraph $G$ the function $d_G$ is not symmetric though it is a \emph{quasi-metric}, as it satisfies the triangle inequality and the identity property. Hence segments $[xy]^G$ and $[yx]^G$ may be different. However in a graph the function $d_G$ is symmetric and thus it is a metric.

The \emph{line} defined by the pair $(x,y)$, denoted by $\ovd{xy}^G$, is the set of all vertices $z$ such that there is a shortest path containing $x$, $y$ and $z$.
Then,
$$\ovd{xy}^G=\{z\in V\mid x\in [zy]^G\lor z\in [xy]^G \lor y\in [xz]^G\}.$$
As above, one must keep in mind that for graphs we have that $\ovd{xy}^G=\ovd{yx}^G$, but this is not valid for arbitrary digraphs.

The set of all lines defined by vertices of a digraph $G=(V,A)$ will be denoted by $\cL(G)$. That is,
$$\cL(G)=\{\ovd{xy}^G\mid x,y\in V, x\neq y\}.$$
Then a digraph is thin if and only if $|\cL(G)|<|V|$.

 To ease the presentation we drop the reference to the digraph whenever it is clear from the context. Also, for two vertices $a$ and $b$ of a digraph we denote by $a\to b$ the fact that $d(a,b)=1$ and by $a \too b $ the fact that $a\to b$ and $d(b,a)>1$. In some places we make the abuse of notation $a\too b \to c$ to shorten $a\too b$ and $b\to c$.

 Our first result gives us conditions under which two lines with the same initial point or with the same end point are different.

\begin{lemma}\label{l:samestart}
 Let $G=(V,A)$ be a digraph.
 Let $x,y$ and $z$ in $V$ such that $d(x,y)=d(x,z)$ and $d(z,x)+d(x,z)>d(z,y)$, then
 $z\notin \ovd{xy}$. Similarly, if $d(y,x)=d(z,x)$ and $d(z,x)+d(x,z)>d(y,z)$, then $z\notin \ovd{yx}$.
\end{lemma}
\begin{proof}
 Notice that since $d(x,y)=d(x,z)$ we get that $y\notin [xz]$ and $z\notin [xy]$. Hence, $z\in \ovd{xy}$ if and only if $x\in [zy]$. Since
$$d(z,x)+d(x,y)=d(z,x)+d(x,z)>d(z,y),$$
we get that $x\notin [zy]$ and thus $z\notin \ovd{xy}$.

For the second statement we have that $y\notin [zx]$ and $z\notin [yx]$ which implies that $z\in \ovd{yx}$ if and only if $x\in [yz]$. As before, this is not possible as 
$$d(y,x)+d(x,z)=d(z,x)+d(x,z)>d(y,z).$$ 
 \end{proof}

For a set of vertices $U$ not containing $x$ we denote by  $(x,U)$ and $(U,x)$ the following sets of lines:
$$
(x,U)=\{\ovd{xy}\mid y\in U\}\text{ and }
(U,x)=\{\ovd{yx}\mid y\in U\}.
$$
The set $(x,V\setminus\{x\})$ will be denoted by $\cL^x$.

\section{Thin oriented graphs of diameter two}\label{s:digraph}

Let $G$ be an oriented graph of diameter two. Then when $d(x,y)=1$ we have that $d(y,x)=2$. We also have that the line $\ovd{xy}$ defined by two distinct vertices $x$ and $y$ at distance two is exactly the segment $[xy]$ defined by these two vertices. In fact, if $d(x,y)=2$, then for any $z\notin \{x,y\}$, if $x\in [zy]$ we get that $d(z,y)\geq 3$ and if $y\in [xz]$, then $d(x,z)\geq 3$. Hence, $z\in \ovd{xy} \iff z\in [xy]$, when $d(x,y)=2$.

In this section we prove that $\overrightarrow{C_3}$ is the only thin oriented digraph of diameter two.

\begin{theorem}\label{t:oriented2}
Let $G=(V,A)$ be an oriented graph of diameter $2$. If $G$ is thin, then $G$ is isomorphic to $\overrightarrow{C_3}$.
\end{theorem}

\begin{proof}
Let $G=(V,A)$ be a thin oriented graph of diameter two. We prove that $G$ is $\overrightarrow{C_3}$.

Since $G$ has diameter 2, for each vertex $x$ we have that
$$V=\{x\}\cup N(x) \cup N^2(x),$$
where $N(x)=\{y\mid d(x,y)=1\}$ and $N^{2}(x)=\{y\mid d(x,y)=2\}$. Thus,
$$\cL^x=(x,N(x))\cup (x,N^2(x)).$$

As $G$ is an oriented graph we have that two distinct vertices $u$ and $v$ satisfy $$d(u,v)+d(v,u)\geq 3.$$

From Lemma \ref{l:samestart} we get that $|(x,N(x))|=|N(x)|$ and $|(x,N^2(x))|=|N^2(x)|$. 

When $N(x)=\{u\}$, then $N(u)=N^2(x)$ and $\overrightarrow{xu}=V$. Moreover, $|(x,N^2(x))|=|V|-2$.

Let us assume that $G$ is not $\overline{C_3}$, then there are $v,v'\in N^2(x)$. Then, $\overrightarrow{xu}\notin (x,N^2(x))$. As $G$ is an oriented graph of diameter two we can assume that $d(v,v')=2$ and then $x\notin \ell:=\overrightarrow{vv'}$ which proves that $\ell\notin (x,N^2(x))\cup \{\overrightarrow{xu} \}$ and then $G$ is not thin.

If $(x,N(x))\cap (x,N^2(x))$ is empty we get that $|\cL^x|=|V|-1$. 
As we can assume that $|N(x)|\geq 2$, we get that $G$ is not thin since we have the following property.

\begin{claim}\label{cl:linedefinaplus}
For each $x\in V$ and each $a,a'\in N(x)$ with $a\neq a'$, $x\notin \ovd{aa'}\cup \ovd{a'a}$.
\end{claim}
\begin{proof} By interchanging the roles of $a$ and $a'$, it is enough to prove that $x\notin \ovd{aa'}$.
If $a\to a'$, then from Lemma \ref{l:samestart} we get that $x\notin \ovd{aa'}$. Otherwise, $d(a,a')=2$ and $\ovd{aa'}=[aa']$, since $G$ has diameter two. As $d(a,a')=2<d(a,x)+d(x,a')$, since $d(x,a')+d(a,x)=1+d(a,x)=3$, we get that $x\notin [aa']=\ovd{aa'}$.
\end{proof}

We now study what happens when $(x,N(x))\cap (x, N^2(x))$ is non-empty.
We can continue under the assumption that for every $x\in V$,
$(x,N(x))\cap (x, N^2(x))$ is non-empty and $|N(x)|\geq 2$.

If $\ovd{xy}=\ovd{xz}$ with $x\to z$, and $z\neq y$, then we know from Lemma \ref{l:samestart} that $\ovd{xy}=\ovd{xz}=\{x,z,y\}$. Moreover we have that  $y\to x$ as we prove below.

\begin{claim}\label{cl:diam2rep} For distinct vertices $x,y,z\in V$, if $x\to z$ and $\ovd{xy}=\ovd{xz}$, then $d(x,y)=2$ and $y\to x$.
 \end{claim}
\begin{proof}
We already know that $\ovd{xy}=\ovd{xz}=\{x,z,y\}$ and that $x\to z$. Since $G$ has diameter two and directed girth at least three, for each $a\in N(x)$, $d(a,x)=2$.
Hence, $d(z,x)=2$. Thus, there is a vertex $w$ such that $d(z,w)=d(w,x)=1$ and $w\notin N(x)$. This implies that $w\in N^2(x)$. Hence, $w\in \ovd{xz}=\{x,z,y\}$ which shows that $w=y$. We conclude that $y\to x$.
\end{proof}

Let
$$V^x_{12}=\{(z,y)\in V^2\mid z\in N(x), y\in N^{2}(x) \textrm{ and } \ovd{xz}=\ovd{xy}\},$$

$$V^x_1=\{z\in V\mid z\in N(x) \text{  and } \forall y\in V, y\neq z, \ovd{xz}\neq \ovd{xy}\}$$
and
$$V^x_2=\{y\in V\mid y\in N^2(x) \text{  and } \forall z\in V, z\neq y, \ovd{xy}\neq \ovd{xz}\}.$$

Then,
$$|V|=2|V^x_{12}|+|V^x_1|+|V^x_2|+1 \text{ and }|\cL^x|=|V^x_{12}|+|V^x_1|+|V^x_2|.$$

From Claim \ref{cl:diam2rep} we know that the set $V^x_{12}$, for any $x\in V$, is given by
$$V^x_{12}=\{(z_1,y_1),\ldots,(z_k,y_k)\},$$
where, for $i=1,\ldots,k$, we have that
$\ovd{xz_i}=\ovd{xy_i}$, $x\to z_i$ and $d(x,y_i)=2$. Moreover, we also know that for each $i=1,\ldots,k$, $\ovd{xz_i}\cap N(x)=\{z_i\}$, $\ovd{xz_i}\cap N^2(x)=\{y_i\}$ and $y_i\to x$. Hence, for $1\leq i,j\leq k$, with $i\neq j$, $d(z_i,y_j)=2$.

Moreover,
$$|\cL^x|=|V|-1-k.$$
To finish the proof we show that the set $\cL(G)\setminus \cL^x$ contains at least $k+1$ lines.

We first consider the case when $k\geq 2$.

\begin{claim}\label{cl:tournament}
$G[\{z_1,\ldots,z_k\}]$ and $G[\{y_1,\ldots,y_k\}]$, the digraphs induced by $\{z_1,\ldots,z_k\}$ and $\{y_1,\ldots,y_k\}$, respectively,  are tournaments, and $z_iz_j\in A \iff y_jy_i\in A$.
\end{claim}
\begin{proof} Let $i,j\in \{1,\ldots,k\}$ with $i\neq j$.
We know that $d(z_i,y_j)=2$. Let $w\in [z_iy_j]$. Then, when $w\in N(x)$ we get that $w=z_j$ and when $w\in N^2(x)$ we get that $w=y_i$. Hence, if $d(z_i,z_j)=2$, then $y_i\to y_j$.
The same argument applied to $z_j$ and $y_i$ shows that if $d(z_j,z_i)=2$, then $y_j\to y_i$.

Assume that $d(z_i,z_j)=2$. Then $y_i\to y_j$. Since $G$ is an oriented graph, we have that $d(y_j,y_i)>1$ which implies that $d(z_j,z_i)\neq 2$. Thus, $z_j \to z_i$ because $G$ has diameter two.
Therefore,  we conclude that $G[\{z_1,\ldots,z_k\}]$ is a tournament.
Similarly, if $d(y_i,y_j)>1$, then $d(z_j,z_i)=1$ which implies that $d(z_i,z_j)=2$. In turn this implies that $y_i\to y_j$. Hence $G[\{y_1,\ldots,y_k\}]$ is also a tournament. Clearly,  $z_i\to z_j \iff y_j\to y_i$.
\end{proof}

\begin{claim}
Let
$$(\cL^x)':=\{\ovd{z_iz_j}\mid z_i \to z_j, i,j\in\{1,\ldots,k\}\}.$$ Then, no line in $(\cL^x)'$ contains the vertex $x$ and $|(\cL^x)'|=\binom{k}{2}$.
\end{claim}
\begin{proof}
From Claim \ref{cl:linedefinaplus} we get immediately that $x\notin \ovd{z_iz_j}$, when $z_i\to z_j$. Hence, no line in $(\cL^x)'$ contains the vertex $x$.

Now we prove that $|(\cL^x)'|=\binom{k}{2}$.
Let us assume that $z_1\to z_2$.  Let $i$ and $j$ be such that $z_i\to z_j$ and $\{1,2\}\neq \{i,j\}$. If $i=1$ or $2=j$ from Lemma \ref{l:samestart} we get that $\ovd{z_1z_2}\neq \ovd{z_iz_j}$. If $i=2$, then $y_2\in \ovd{z_1z_2}$ but $y_2\notin \ovd{z_2z_j}$, by Lemma \ref{l:samestart}. Hence, we get that $i\notin \{1,2\}$.

From Claim \ref{cl:tournament} we can assume that $z_1\to z_i$. Then, $z_2\notin [z_1z_i]$, $z_i\notin [z_1z_2]$, and $z_1\notin [z_iz_2]$, as otherwise, $d(z_i,z_2)\geq 3$.
\end{proof}

From previous claim we have that the sets $\cL^x$ and $(\cL^x)'$ are disjoint and that $|(\cL^x)'|=k(k-1)/2\geq k+1$, when $k\geq 4$. Then,
we can continue the proof under the assumption that $k\in \{2,3\}$.
Let $z_1,z_2$ be such that $z_1\to z_2$.
We prove that $\ovd{z_2y_2}\notin \cL^x\cup (\cL^x)'$. In fact, since $d(z_2,x)=d(z_2,y_2)+d(y_2,x)$, the line $\ovd{z_2y_2}$ contains the vertex $x$ which shows that $\ovd{z_2y_2}\notin (\cL^x)'$.
It also contains $z_1$ and $y_1$ because  $d(z_1,y_2)=d(z_2,y_1)=2$ and from Claim \ref{cl:tournament} we have that  $y_2\to y_1$.
Then, $z_1,y_1,z_2,y_2\in \ovd{z_2y_2}$. Therefore, $\ovd{z_2y_2}$ contains two vertices in $N(x)$ and two vertices in $N^2(x)$ which implies that $\ovd{z_2y_2}\notin \cL^x$.
Hence, when $k=3$, we have that 
$$|\cL^x\cup (\cL^x)'\cup \{ \ovd{z_2y_2}\}|=|V|-1-k+k(k-1)/2+1=|V|.$$

Let $k = 2$. We can assume that $z_1\to z_2$ and $(\cL^x)'=\{\ovd{z_1z_2}\}$. Previous counting shows that we have $|V|-1$ lines in $\cL^x\cup (\cL^x)'\cup \{\ovd{z_2y_2}\}$.
An additional line is $\ovd{y_2y_1}$.
Indeed, as $y_2\to x$ we know that $x$ is not in
$\ovd{y_2y_1}$, by
Lemma \ref{l:samestart}, which implies that $\ovd{y_2y_1}\notin \cL^x \cup \{\ovd{z_2y_2}\}$. By the same lemma we get that
$y_1$ is not in $\ovd{z_1z_2}$, since $z_1\to y_1$. Hence, $\ovd{y_2y_1}\notin (\cL^x)'$ either.

Finally, we consider the case $k=1$. In this situation we have that
$|\cL^x|=|V|-2$.
Let us assume that $G$ is not $\overrightarrow{C_3}$. Then the set $(N(x)\cup N^2(x))\setminus \{z_1,y_1\}$ is not empty. If there is $z\in N(x)$, $z\neq z_1$, then there is $y\in N^2(x)$ with $z\to y$, $y\neq y_1$. Similarly, if there is $y\in N^2(x)$, then there is $z\in N(x)$ with $z\to y$, $z\neq z_1$.

Therefore, $|N(x)|,|N^2(x)|\geq 2$. We consider the lines
$\ovd{z_1y}$ and $\ovd{zy_1}$ for $z\in N(x)$ and $y\in N^2(x)$.
Then, $d(z_1,y)=d(z,y_1)=2$ and $x\notin \ovd{z_1y}\cup \ovd{zy_1}$.
To finish we prove that $\ovd{z_1y}\neq \ovd{zy_1}$. In fact,
if $z\in \ovd{z_1y}=[z_1y]$ we have that $z_1\to z$ which implies that $d(z,z_1)=2$. Hence, $z_1\notin \ovd{zy_1}=[zy_1]$.

\end{proof}

\section{Thin bridgeless digraphs of diameter at most three}

In this section, we study two subclasses of bridgeless digraphs of diameter three. On the one hand,  bipartites digraphs and, on the other hand, oriented graphs of directed girth four.
\subsection{Bipartite digraphs}
As far as we know there are only two thin bridgeless bipartite digraphs: $K_{2,2}$ and $K_{2,3}$. In \cite{MZ2020} it was proved that these are the only thin bridgeless bipartite \emph{graphs}. With the additional constraint of diameter three we can extend this to the class of bipartite digraphs.
Notice that a bipartite digraph of diameter two is, in fact, a complete bipartite graph. Hence, we can focus on bipartite digraphs of diameter three.

\begin{proposition}\label{p:bip9ormore}
 A thin bipartite digraph of diameter three has at most eight vertices.
 \end{proposition}
\begin{proof}

Let $G=(X\cup Y, A)$ be a bipartite digraph of diameter three
and let $p=|X|$ and $q=|Y|$ with $p\geq q\geq 1$, and $n:=p+q\geq 9$.
We prove that $G$ is not thin.

We have that for each $(i,j) \in A$, either $i\in X$ and $j\in Y$, or $i\in Y$ and $j\in X$.

Since $G$ is bipartite, the bound on the diameter implies that for every two distinct vertices $u$ and $v$ with $|X\cap\{u,v\}|=|Y\cap \{u,v\}|=1$,  we have that
$$ (d(u,v),d(v,u))\in \{(1,1),(1,3),(3,1),(3,3)\},$$ and if $u,v\in X$ or $u,v\in Y$ we have that $(d(u,v),d(v,u))=(2,2)$.

From this, it follows  that the only vertices in $\ovd{uv} \cap X$ are $\{u,v\}$, when $u,v\in X$, and similarly,
$\ovd{uv}\cap Y=\{u,v\}$, when $u,v\in Y$. Therefore, the set of different lines defined by two vertices on the set $X$ (resp. $Y$), denoted by $\cal X$ (resp. $\cal Y$),
has at least $p(p-1)/2$ (resp. $q(q-1)/2$) lines. As the digraph has at least nine vertices we have that $p\geq 5$.
Then, we have that $p(p-1)\geq 4p$ which implies that $|\mathcal{X}|\geq p(p-1)/2\geq 2p\geq p+q=n$.
\end{proof}

%

We complement Proposition \ref{p:bip9ormore} with the next theorem, where we prove that in fact, $C_4$ and $K_{2,3}$ are the only thin bridgeless bipartite digraphs of diameter at most three.

\begin{theorem}\label{t:bip}
The graphs $C_4$ and $K_{23}$ are the only thin bridgeless bipartite digraphs with diameter at most three.
\end{theorem}
\begin{proof}
Let $G=(X\cup Y, E, A)$ be a bridgeless bipartite digraph of diameter at most three. By the same obsevation we made above, we can assume that $G$ has diameter three.
Let $p=|X|$ and $q=|Y|$ with $p\geq q\geq 1$, and $n:=p+q$.
From the proof of Proposition \ref{p:bip9ormore} we can assume that $p\leq 4$.

Since the digraph is bridgeless we can assume that $G$ has no vertex with in-degree or out-degree one and then $4\geq p\geq q\geq 2$. If $G$ is a bipartite graph, then the result follows from Theorem 4 in \cite{MZ2020}. Then we can assume that there is $uv\in A$ with $vu\notin A$. That is, $d(u,v)=1$ and $d(v,u)=3$. Let $A^*\subseteq A$ and $\mathcal{A}$ be given by:
$$A^*=\{(u,v)\in A \mid (v,u)\notin A\}\text{ and }\mathcal{A}=\{\ovd{uv}\mid (u,v)\in A^*\}.$$

Then $\mathcal{A}\neq \emptyset$ and $\mathcal{A}\cap (\mathcal{X}\cup \mathcal{Y})=\emptyset$, where $\mathcal{X}$ and $ \mathcal{Y}$ are the set of lines defined in the proof of Proposition \ref{p:bip9ormore}.
In fact, let $(u,v)\in A^*$. Since $G$ has no triangle we know that
$N(v)\cup \{w\mid (w,u)\in A\}\subseteq \ovd{uv}, $
$u\notin N(v)$ and $(v,u)\notin A$. Since the in-degree of $u$ and the out-degree of $v$  are at least two we get that  $|\ovd{uv}\cap X|$ and $|\ovd{uv}\cap Y|$ are at least 3  which implies that $\ovd{uv}\notin \mathcal{X}\cup \mathcal{Y}$.
Notice that by the same argument we can assume that $q\geq 3$.

Then, $G$ has at least $|\mathcal{X}\cup \mathcal{Y}|+|\mathcal{A}|$ lines.
This shows that when $p=4$ and $q=3$,
the digraph $G$ has at least $6+|\mathcal{A}|\geq 7$, since $\mathcal{A}\neq \emptyset$.

We now focus on the case $p=q=4$. Again, as $\mathcal{A}\neq \emptyset$, if $|\mathcal{X}\cup \mathcal{Y}|\geq 7$ we are done. Hence, we assume that for each two vertices $x$ and $x'$ in $X$, $\ovd{xx'}=\ovd{x'x}$ and that $\mathcal{X}=\mathcal{Y}$

If for some vertex $v$ there are two vertices $u,u'$ with $(u,v),(u',v)\in A^*$, then  $u'\notin \ovd{uv}$ and thus $\ovd{u'v}\neq \ovd{uv}$. Hence, $|\mathcal{A}|\geq 2$ and we are done. Therefore, we can assume that for each vertex $v$, $|\{u\mid (u,v)\in A^*\}|\leq 1$.

Let $(u,v)\in A^*$. W.l.o.g. we can assume that $u\in X$ and $v\in Y$.
As the out-degree of $u$ it at least two, there is $v'\neq v$ such that $d(u,v')=1$. Since $u\to v'$ and $u\to v$ we have that $v\notin [uv']$ and $v'\notin [uv]$. Moreover, as $d(v,u)=3$ and the diameter of $G$ is three, we get that $u\notin [vv']$. Therefore, $v\notin \ovd{uv'}$.

If $\ovd{uv'}\in \mathcal{X}$, then there is $u'\neq u$ such that $d(u,u')=2$ and $\ovd{uu'}=\ovd{u'u}=\ovd{uv'}$. We get a contradiction by proving that $(u',v)\in A^*$.

Since $v\notin \ovd{uv'}$ we have that $d(v,u')=3$ as otherwise $d(v,u')=1$ and $v\in \ovd{uu'}$. Similarly, we have that $d(u',v)=1$, as otherwise, $d(u',v)=3$ which implies that $3=d(u',v)=d(u',u)+d(u,v)$ and then $v\in \ovd{u'u}$. Therefore, $(u',v)\in A^*$.

Finally, we consider the case $p=q=3$. Let $(u,v)\in A^*$, $X=\{u,u',u''\}$ and $Y=\{v,v',v''\}$. Since the out-degree of $v$ and the in-degree of $u$ are at least two we get that $d(v,u')=d(v,u'')=1=d(v',u)=d(v'',u)$.
Then, $\ovd{uv}=X\cup Y$ is an universal line.
Since the out-degree of $u$ and the in-degree of $v$ are at least two we can assume that $d(u',v)=1$ and $d(u,v')=1$.

We prove that $\{v,v'\}\subseteq \ovd{uu'}\cap \ovd{uu''}$. For $d(u,v)=d(u',v)=d(v,u')=d(v,u'')=1$ we get that $v\in \ovd{uu'}\cap \ovd{uu''}$.
If $d(v',u')=1$, then $u\to v'\to u'$ and thus $v'\in \ovd{uu'}$. Otherwise, $d(v',u')=3$ since $G$ has diameter three. Then, $v'\to u \to v \to u'$ is a shortest path of length three which shows that $v'\in \ovd{uu'}$. Similarly, if $d(v',u'')=1$, then $u\to v'\to u''$ and thus $v'\in \ovd{uu''}$. Otherwise, $d(v',u'')=3$.
Then, $v'\to u \to v \to u''$ is a shortest path of length three. Again we conclude that $v'\in \ovd{uu''}$.

The line $\ovd{u'v}$ does not contain $u$. In fact, since $u\to v$, $u'\notin [uv]$,
$u'\to v$, then $u\notin [u'v]$ and since $d(v,u)=3$  and $G$ has diameter three, we get that  $v\notin [u'u]$.

We prove that there are at least five different lines containing $u$. To this end we consider the following seven lines, all containing $u$.
$$\ovd{uv},\ovd{v'v},\ovd{vv'},\ovd{v''v},\ovd{v''v'},\ovd{uu'},\ovd{uu''}.$$

Since among them only $\ovd{uv}$ is universal, it is enough to prove that in this collection  there are at least four different non-universal lines. If $\ovd{vv'}\neq \ovd{v'v}$, then the four lines $\ovd{vv'},\ovd{v'v},\ovd{vv'}$ and $\ovd{v''v'}$ are different.

Otherwise, $\ovd{vv'}=\ovd{v'v}$ and, since $\ovd{uu'}\neq \ovd{uu''}$ and $\{v,v'\}\subseteq \ovd{uu'}\cap \ovd{uu''}$, either $\ovd{uu'}$ or $\ovd{uu''}$ does not belong to  $\{\ovd{v'v}=\ovd{vv'},\ovd{v''v},\ovd{v''v'}\}$. We conclude that
$$|\{\ovd{v'v},\ovd{v''v},\ovd{v''v'},\ovd{uu'},\ovd{uu''}\}|\geq 4.$$

\end{proof}

\subsection{Digraphs of directed girth four.}

From Theorem \ref{t:bip} we know that thin bridgeless digraphs of diameter three (if any) are not bipartite. We now prove that they can not have directed girth four.

\begin{theorem}\label{t:briddiam3girth4}
Let $G=(V,A)$ be a bridgeless digraph of diameter three and directed girth four.
Then, $|\cL(V)|\geq |V|$.
\end{theorem}

\begin{proof}

For a given vertex $x$, we shall consider the following sets of lines:
$\cL^x_i=(x,N^i(x))$,  where $N^i(x)=\{z\in V\setminus \{x\}\mid d(x,z)=i\}$, for $i=1,2,3$. 
Hence, 

From Lemma \ref{l:samestart} it follows that
\begin{claim}\label{cl:uniqusamelevel}
For each $i=1,2,3$ and each $a\in N^i(x)$, $\ovd{xa}\cap N^i(x)=\{a\}$.
\end{claim}

This implies that the number of lines in $\cL_i=(x,N^i(x))$ is exactly $|N^i(x)|$, for each $i=1,2,3$.
As in the proof of Theorem \ref{t:oriented2}, when these sets are pairwise disjoint we have that the set $\mathcal{L}^x$ defined by
$$\mathcal{L}^x:=\cL^x_1\cup\cL^x_2\cup \cL^x_3$$ has $|V|-1$ lines and we need only one additional line to conclude. In most cases, we shall prove that a line does not belong to $\mathcal{L}^x$ by proving that $x$ does not belong to it. In the next claim we present some of these lines.

\begin{claim}\label{l:wtx}
 Let $x,a,u,v\in V$ with $d(x,a)=d(x,u)=1$, $d(a,v)=3$ and $v\neq x$.
 Then, $x\notin \ovd{au}\cup \ovd{av}$. 
\end{claim}
\begin{proof} Notice that since $d(x,a)=1$, then $d(a,x)=3$
and thus $x\notin [ay]$, for each $y\neq x$.
This immediately implies that $x\notin \ovd{av}$ because $x\neq v$ and $\ovd{av}=[av]$, when  $d(a,v)=3$.
 Since $d(x,a)=d(x,u)=1$ we get that $x\in \ovd{au}$ if and only if $x\in [au]$ which as before is not possible because $d(a,x)=3$.
\end{proof}

For $x\in V$, let $R^x$ be the following set of pairs of vertices.
$$R^x=\{(a,c)\mid d(x,a)=1, d(x,c)=3 \text{ and } d(a,c)=3\}.$$
We shall consider the following set $\mathcal{L}_{R^x}$ of lines.
$$\mathcal{L}_{R^x}=\{\ovd{ac}\mid (a,c)\in R^x\}.$$
From previous claim we know that no line in $\mathcal{L}_{R^x}$ contains $x$. Hence, $\mathcal{L}^x\cap \mathcal{L}_{R^x}=\emptyset$.

We have that:

\begin{claim}\label{l:linesplus}
For each $x\in V$, $|\cL_{R^x}|=|R^x|$
\end{claim}
\begin{proof} Let $x,a,a',c,c'\in V$ with $(a,c),(a',c')\in R^x$. We prove that if $a\neq a'$ or $c\neq c'$, then $\ovd{ac}\neq \ovd{a'c'}$.

 Notice that, as $d(a,c)=d(a',c')=3$, the lines $\ovd{ac}$ and $\ovd{a'c'}$ are the segments $[ac]$ and $[a'c']$, respectively.

By Lemma \ref{l:samestart}, we know that if $a=a'$, then $\ovd{ac}\neq \ovd{ac'}$, because in this case $c\neq c'$ by hypothesis and then $d(a,c)+d(c,a)\geq 4> 3\geq d(c,c')$. Then we continue under the assumption $a\neq a'$. For the sake of a contradiction let us assume that $[ac]=[a'c']$. From  $a'\in [ac]$ we get that $d(a,a')+d(a',c)=3$ which implies that $d(a,a')=1$ and $d(a',c)=2$, because $3=d(x,c)\leq d(x,a')+d(a',v)=1+d(a',c)$. Similarly, from $a\in [a'c']$ we get that $d(a',a)+d(a,c')=3$ and then $d(a',a)=1$ which together with $d(a,a')=1$ is a contradiction.
 \end{proof}

For each $x\in V$, we consider the following definitions.
$$V^x_{123}=\{(t_1,t_2,t_3)\in V^3\mid \ovd{xt_1}=\ovd{xt_2}=\ovd{xt_3}, t_i\in N^i(x),  i=1,2,3\},$$
$$V^x_{12}=\{(t_1,t_2)\in V^2\mid \ovd{xt_1}=\ovd{xt_2} \land d(x,t_i)=i, i=1,2 \land \forall t_3\in V, (t_1,t_2,t_3)\notin V_{123}\},$$
$$V^x_{23}=\{(t_2,t_3)\in V^2\mid \ovd{xt_2}=\ovd{xt_3} \land d(x,t_i)=i, i=2,3\land \forall t_1\in V, (t_1,t_2,t_3)\notin V_{123}\},$$
$$V^x_{13}=\{(t_1,t_3)\in V^2\mid \ovd{xt_1}=\ovd{xt_3} \land d(x,t_i)=i, i=1,3 \land \forall t_2\in V, (t_1,t_2,t_3)\notin V_{123}\}$$

and
$$V^x_{i}=\{y\in V\mid d(x,y)=i \land \forall z\neq y,x, \ovd{xy}\neq \ovd{xz}\},$$
for $i=1,2,3$.

To ease the presentation, for each $x\in V$ and each $t\in V^x_{123}\cup V^x_{13}\cup  V^x_{12}\cup V^x_{23}$ we shall denote by $t_i$ the coordinate of $t$ with $d(x,t_i)=i$. 
For each $x\in V$, the definitions of the sets $V^x_{123}, V^x_{13},V^x_{12}$ and $V^x_{23}$ imply that a given vertex can be a coordinate of at most one of their elements.  That is, for each $t,r\in  V^x_{123}\cup V^x_{13}\cup  V^x_{12}\cup V^x_{23}$ we have that if there is an index $i$ such that $t_i=r_i$, then $r=t$. Hence, for each $x\in V$,
$$|V|=3|V^x_{123}|+2(|V^x_{12}|+|V^x_{23}|+V^x_{13}|)+|V^x_1|+|V^x_2|+|V^x_3|{+1}.$$

Moreover,
\begin{equation}\label{e:main}
\begin{split}
 |\mathcal{L}^x|& = |V^x_{123}|+|V^x_{12}|+|V^x_{23}|+|V^x_{13}|+|V^x_1|+|V^x_2| \\
& =  |V|-1-2|V^x_{123}|-|V^x_{12}|-|V^x_{23}|-|V^x_{13}|.
\end{split}
\end{equation}

To achieve our goal we have to find two additional lines for each line associated to $V^x_{123}$ and one additional line for those associated to $V^x_{ij}$, for $i\neq j$. Beside that, we need to find an extra line to compensate the $-1$.

\begin{claim}\label{cl:relations}
Let $i,j,k$, $i<j$, such that $\{i,j,k\}=\{1,2,3\}$. Then, for each $x\in V$,  for each $t\in V^x_{ij}$,
we have that  $\emptyset \neq \ovd{xt_i}\cap N^k(x)\subseteq V^x_k$.
\end{claim}
\begin{proof} When $j=3$, for $t\in V^x_{ij}$ a shortest path from $x$ to $t_j$ contains a vertex in $N^k(x)$. When $(i,j)=(1,2)$, a shortest path from $t_j$ to $x$ contains a vertex in $N^3(x)$.

If a vertex $z\in \ovd{xt_i}\cap N^k(x)$ is such that $z\notin V^x_k$, then there is a vertex  $r$ in $V^x_{123}\cup V^x_{12}\cup V^x_{23}\cup V^x_{13}$ with $r_k=z$. But then either $t_i=r_i$ or $t_j=r_j$ which implies that $\ovd{xt_i}=\ovd{xr_i}=\ovd{xr_k}$ or $\ovd{xt_j}=\ovd{xr_j}=\ovd{xr_k}$ contradicting the fact that $t\in V^x_{ij}$ (and not to $V^x_{123}$).
\end{proof}

\begin{claim}\label{cl:Rx} For each $x\in V$ we have that
\begin{equation}\begin{split}
 |R^x|\geq & (|V^x_{123}|+|V^x_{13}|)(|V^x_{123}|+|V^x_{13}|-1)+\\
 &(|V^x_{123}|+|V^x_{13}|)(|V^x_{12}|+|V^x_{23}|+|V^x_1|+|V^x_3|)
\end{split}
\end{equation}
\end{claim}
\begin{proof}
 Let $t\in V^x_{123}\cup V^x_{13}\cup  V^x_{12}\cup V^x_1$ and $r\in V^x_{123}\cup V^x_{13}\cup V^x_{23}\cup V^x_3$, $r\neq t$. We show that $d(t_1,r_3)=3$. In fact, if $d(t_1,r_3)\leq 2$, then $d(t_1,r_3)=2$. Hence, $r_3\in \ovd{xt_1}$ and $t_1\in \ovd{xr_3}$. If $t\in V^x_{123}\cup V^x_{13}$, this implies that $t_3=r_3$ and thus $t=r$. When $t\in V^x_{12}$, the fact that $t_1\in \ovd{xr_3}$ implies that $t_2=r_2$ and again we get $t=r$.
\end{proof}

\begin{claim}\label{cl:12313} If there is $x\in V$ such that $V^x_{123}\cup V^x_{13}$ is non-empty, then $|\cL(V)|\geq |V|$.
\end{claim}
\begin{proof}
 From previous claim we have that 
 \begin{equation*}
 \begin{split}
  |\cL(V)|=& |\cL(V)\setminus (\cL^x \cup \cL_{R^x})|+|\cL^x|+|R^x| \\
  \geq &|\cL(V)\setminus (\cL^x \cup \cL_{R^x})|+|V|-1-2|V^x_{123}|-|V^x_{12}|-|V^x_{13}|-|V^x_{23}|+\\
  & (|V^x_{123}|+|V^x_{13}|)(|V^x_{123}|+|V^x_{13}|-1)+\\
  &(|V^x_{123}|+|V^x_{13}|)(|V^x_{12}|+|V^x_{23}|+|V^x_1|+|V^x_3|)
 \end{split}
\end{equation*}

Previous inequality can be written as follows.

\begin{equation}\label{e:maintotal}
 \begin{split}
  |\cL(V)|\geq &|V|+|\cL(V)\setminus (\cL^x \cup \cL_{R^x})|-1+(|V^x_{123}|+|V^x_{13}|-2)(|V^x_{123}|+|V^x_{13}|-1)+\\
  &|V^x_{13}|-2+(|V^x_{123}|+|V^x_{13}|-1)(|V^x_{12}|+|V^x_{23}|+|V^x_1|+|V^x_3|)\\
  &|V^x_1|+|V^x_3|
 \end{split}
\end{equation}

It worth to notice that for each $t\in V^x_{123}\cup V^x_{12}$, $t_2$ is the only out-neighbor of $t_1$ in $N^2(x)$. As $G$ is bridgeless, $t_1$ must have at least another out-neighbor in $N(x)$. This implies that $V^x_1\cup V^x_{13}$ is non-empty when $|V^x_{123}|+|V^x_{12}|\leq 3$, because the directed girth of $G$ is four.

We also have that for each $i\in \{1,2,3\}$, $v$ belongs to $N^i(x)$ if and only if there is $t\in V^x_{123}\cup V^x_{kj}\cup V^x_i$, when $i\in \{j,k\}$ such that $v=t_i$.

When $|V^x_{123}|+|V^x_{13}|\geq 4$ we get that $|\cL(V)|\geq |V|$. 
When $|V^x_{123}|+|V^x_{13}|=3$, then we also get $|\cL(V)|\geq |V|$. In fact, if $|V^x_{13}|=0$, then the above comment implies that $V^x_{12}\cup V^x_1$ is non-empty.

When $|V^x_{123}|=2$ and $V^x_{13}$ is empty, then we have that $|V^x_1|\geq 1$ or $|V^x_{12}|\geq 2$, either case we have problems only when  $V^x_3\cup V^x_{23}$ is empty. But when $V^x_{12}$ is not empty, then $V^x_3$ is not empty either, since the path from $t_1$ to $x$, for $t\in V^x_{12}$, contains a vertex in $V^x_3$. Therefore, we can assume that $V^x_{12}$ is empty. Hence $V^x_1$ is non-empty and then $V^x_3\cup V^x_{23}$ is non-empty, since a path from $z\in V^x_1$ to $x$ of length three contains a vertex in $N^3(x)$ which can not be part of a tuple in $V^x_{123}$. 

Let $|V^x_{123}|+|V^x_{13}|=2$ and $V^x_{13}$ be non-empty.
For $r\in V^x_{13}$ we have that $\overrightarrow{xr_1}\cap N^2(x)$ contains at least two vertices $v$ and $v'$, since $\overrightarrow{xr_1}\notin \cL^x_2$. In this case, the lines $\overrightarrow{r_1v}$ and $\overrightarrow{r_1v'}$ are different as $v\notin \overrightarrow{r_1v'}$, by Lemma \ref{l:samestart}. These lines contain $x$ and $r_3$, since $x\to r_1\to v,v'\to r_3$, which implies that they do not belong to $\cL_{R^x}$ neither to $\cL^x_1\cup \cL^x_3$, since $\overrightarrow{xr_1}=\overrightarrow{xr_3}$ contains both $v$ and $v'$. Moreover, none of them belong to $\cL^x_2$ because, if $\overrightarrow{r_1v}=\overrightarrow{xw}$, then $w=v$ and $\overrightarrow{xv}\cap N(x)=\{r_1\}$, because any in-neighbor of $v$ in $N(x)$ belongs to $\overrightarrow{xr_3}$.  Hence, we get the contradiciton $\overrightarrow{xv}=\overrightarrow{xr_1}$.

    It remains to consider the case when $V^x_{123}\cup V^x_{13}$ is a singleton $\{t\}$.
If $t\in V^x_{123}$, then $t_1$ has an out-neighbor $v$ in $N(x)$ and $V^x_{13}$ is empty.
Let $vv'v''x$ a path of length three from $v$ to $x$. Then, 
$t_1vv'v''$ is a path of length three from $t_1$ to $v''$, which implies that the line $\overrightarrow{vv'}$ contains $x,t_1,v,v',v''$. In turns, it shows that this line does not belong to $\cL_{R^x}$. Moreover, $t_1\notin \overrightarrow{xy},$ for each $y\in \{v,v',v''\}$ and then $\overrightarrow{vv'}\notin \cL^x\cup \cL_{R^x}$.

If $v$ has a out-neighbor $w\in N(x)$, then $w\neq t_1$ and, as before, we can consider a path $ww'w''x$ of length three, from $w$ to $x$. Hence, the line 
$\overrightarrow{ww'}$ contains $x$ and $v$, but $w\notin \overrightarrow{vv'}$ which implies that $\overrightarrow{ww'}\notin \cL^x\cup \cL_{R^x}$
and $\overrightarrow{vv'}\neq \overrightarrow{ww'}$. If $v\in V^x_1$, then we get the conclusion. Otherwise, $v=r_1$, for some $r\in V^x_{12}$, and $v''\in V^x_3$, which also implies the conclusion. 

Therefore, we can assume that no out-neighbor of $t_1$ in $N(x)$ has an out-neighbor in $N(x)$. This implies that each out-neighbor of $t_1$ belongs to $V^x_1$. Hence, we can also assume that $v$ is the only out-neighbor of $t_1$ in $N(x)$, that $V^x_1=\{v\}$ and that $V^x_3$ is empty.  Then, $v''=r_3$ for some $r\in V^x_{23}$ (and thus $v'=r_2$). 
The vertex $v$ must have at least two out-neighbors which implies that $|V^x_{23}|\geq 2$, since $N(x)=\{t_1,v\}$ and $t_1\to v$. Therefore, the line $\overrightarrow{t_1v}$ contains all vertices of $G$ with the exception of $x,t_2$ and $t_3$ and, since $|V^x_{23}|\geq 2$, it contains at least two vertices in $N^3(x)\setminus \{t_3\}$. Whence, it does not belong to $\cL_{R^x}$. This shows that $\overrightarrow{t_1v_1}\notin \cL^x\cup \cL_{R^x}$ and that $\overrightarrow{t_1v_1}\neq \overrightarrow{v_1v_2}$ and we get the conclusion.

Now, let us assume that $t\in V^x_{13}$. Then, $V^x_{123}$ is empty. 
Since $|N(x)|\geq 2$, we have that $V^x_{1}\cup V^x_{12}$ is non-empty, which in turns implies that $V^x_1 \cup V^x_3$ is non-empty as well. To conclude the proof of the claim it is enough to find a line not in $\cL^x\cup \cL_{R^x}$.

Let $z\in N^2(x)$ such that $xt_1zt_3$ is a path of length three from $x$ to $t_3$. 
The line $\overrightarrow{t_1z}$ contains the vertices $x,t_1,z,t_3$. Then, it does not belong to $\cL_{R^x}$. Let us assume that there is $w$ such that $\overrightarrow{xw}=\overrightarrow{t_1z}$. If $w\in N(x)$, then $w=t_1$, and if $w\in N^3(x)$, then $w=t_3$. In either case we get a contradiction with $\overrightarrow{xz}\neq\overrightarrow{xt_1}$. Hence, $w\in N^2(x)$ and then $w=z$.
Let $z'$ be an out-neighbor of $z$ different from $t_3$. If $z'\in N^3(x)$, then 
$z'\in \overrightarrow{xt_1}=\overrightarrow{xt_3}$, which is a contradiction.
If $z'\in N^2(x)$, then $z'\in \overrightarrow{t_1z}=\overrightarrow{xz}$ which is also a contradiction. Hence, $z'\in N(x)$. In this case, $z'\notin \overrightarrow{xz}$ which implies that $z'$ is an out-neighbor of $t_1$. 
We can now proceed as before showing that the line $\overrightarrow{z'v'}$, where $z'v'v''$ is a path of length three from $z'$ to $x$, does not belong to 
$\cL\cup \cL_{R^x}$.
 \end{proof}
 
\begin{claim}\label{cl:1223} If there is $x\in V$ such that $V^x_{12},V^x_{23}\neq \emptyset$, then
$|\cL(V)|\geq |V|$.
\end{claim}
\begin{proof} From Claim \ref{cl:12313} we can assume that $V^x_{123}\cup V^x_{13}$
is empty. Then,

 $$|\cL_x|=|V|-1-|V^x_{12}|-|V^x_{23}|.$$

 For each $p\in V_{12}$ and $q\in V_{23}$ we have that
 $d(p_1,q_3)=3$. Thus, $(p_1,q_3)\in R^x$. Therefore,
 $$|R^x|\geq |V^x_{12}||V^x_{23}|=(|V^x_{12}|-1)(|V^x_{23}|-1)+|V^x_{12}|+|V^x_{23}|-1.$$

Let $v$ be an out-neighbor of $p_1$ in $N(x)$. The line $\overrightarrow{vv'}$, where $vv'v''x$ is a path of length three from $v$ to $x$, contains $x$ and we can prove as before that it does not belong to $\cL^x$.
Hence, the set $\mathcal{L}(V)\setminus (\mathcal{L}^x\cup \mathcal{L}_{R^x})$ is not empty and therefore,
 $$|\mathcal{L}(V)|\geq |V|+(|V^x_{12}|-1)(|V^x_{23}|-1)\geq |V|.$$
\end{proof}

\begin{claim}\label{cl:12v3more2} If there is $x\in V$ such that $V^x_{12}$ is non-empty, then $|\cL(V)|\geq |V|$.
\end{claim}
\begin{proof} From Claims  \ref{cl:12313} and \ref{cl:1223} we can assume that $V^x_{123}=V^x_{13}=V^x_{23}=\emptyset$. Since $G$ has no bridges we can also assume that $|N(x)|\geq 2$, $N^3(x)=V^x_3$ and that this latter set has at least two vertices. Moreover,
$$|\cL^x|=|V|-1-|V^x_{12}|.$$

Let $S=\{t_1\in N(x)\mid t\in V^x_{12}, N^3(x)\subseteq N^2(t_1)\}$. For each $t\in V^x_{12}$ such that $t_1\in S$, we pick $z_t\in N^3(x)$ such that $t_1\notin \overrightarrow{xz_t}$. Such vertex $z_t$ does exist as there is $u\in N(x)$, $u\neq t_1$, and then $z_t$ can be chosen as any vertex in $N^3(x)$ in a shortest path from $u$ to $x$. Let $\cL'$ be the following set of lines.
$$\cL':=\{\overrightarrow{t_1z_t}\mid t\in V^x_{12}, t_1\in S\}.$$
We prove that $\cL'\cap (\cL^x\cup \cL_{R^x})$ is empty. We have that for each $\ell=\overrightarrow{t_1z_t}$, $x\in \ell$, which shows that $\ell\notin \cL_{R^x}$.
By Lemma \ref{l:samestart} and the definition of $S$, we have that $\ell\cap (S\cup N^3(x))=\{t_1,z_t\}$. Again the definition of $S$ implies that $ \overrightarrow{xt_1}$ contains $N^3(x)$. Since this latter set is not $\{z_t\}$ we have that $\ell\neq \overrightarrow{xt_1}$ and by the choice of $z_t$ we have that $\ell\neq \overrightarrow{xz_t}$. This proves that $\ell\notin \cL^x$ since 
any line $\overrightarrow{xw}\in \cL^x$ satisfies that if $w\in N^i(x)$, for some $i\in \{1,2,3\}$, then $\overrightarrow{xw}\cap N^i(x)=\{w\}$. However, $t_1,t_2$ and $z_t$ belong to $N(x)$, $N^2(x)$ and $N^3(x)$, respectively.

The definition of $S$ and Lemma \ref{l:samestart} imply that for $t,r\in V^x_{12}$, with $t_1,r_1\in S$, the lines
$\overrightarrow{t_1z_t}$ and $\overrightarrow{r_1z_r}$ are different. Hence,
$|\cL^x\cup \cL'|=|V|-1-|V^x_{12}|+|S|.$

Let $t\in V^x_{12}$ such that $t_1\notin S$. Then there is $z_t\in N^3(x)$ such that $(t_1,z_t)\in R^x$. Therefore, $$|\cL^x\cup \cL'\cup \cL_{R^x}|\geq |V|-1.$$

Let $(u,z)\in R^x$ and let $uvwz$ a path of length three between $u$ and $z$. Let $\ell=\overrightarrow{uv}$ if $v\in N^2(x)$ and $\ell=\overrightarrow{vw}$ if $v\in N(x)$.  One can see that 
$x,u,v,w,z$ belong to $\ell$ and then $\ell\notin \cL_{R^x}$. When $v\in N(x)$, then $u\notin \overrightarrow{xv}\cup \overrightarrow{xw}\cup \overrightarrow{xz}$ and thus $\ell\notin \cL^x$. 
When $v\in N^2(x)$ we have that $z\notin \overrightarrow{xu}\cup \overrightarrow{xv}$ which implies that $\ell\notin \cL^x_1\cup \cL^x_2$, and $u,v\notin \overrightarrow{xz}$ from which we get that $\ell\notin \cL^x_3$. 

Let $t\in V^x_{12}$, with $t_1\in S$. If $\ell=\overrightarrow{t_1z_t}$, then we know that $\ell\cap (S\cup N^3(x))=\{u,z\}=\{t_1,z_t\}$ which is not possible since $d(t_1,z_t)=2$. Therefore, 
$\ell \notin \cL^x\cup \cL'\cup \cL_{R^x}$ and we get that $|\cL(V)|\geq |V|.$

To finish the proof of the claim we consider the case where $R^x$ is empty. This implies that $|S|=|V^x_{12}|$, that is, for each $t\in V^x_{12}$, $t_1\in S$, and $|\cL\cup\cL'|=|V|-1.$

Since $G$ is bridgeless, the vertex $t_1$ has an out-neighbor $u$ in $N(x)$. The line 
$\ell=\overrightarrow{t_1u}$ does not contain $x$ and then, $\ell\notin \cL\cup \cL'$ which ends the proof.
\end{proof}

In order to finish the proof of the theorem it is enough to prove the following claim.

\begin{claim}\label{cl:23v1more2}
If there is $x\in V$ with $V^x_{23}$ non-empty, then
$\cL(V)|\geq |V|$.
\end{claim}
\begin{proof} As before, from Claims \ref{cl:12313} and \ref{cl:1223} we can assume that $V^x_{123}\cup V^x_{13}\cup V^x_{12}$ is empty.
Now, the proof can be done in a similar way to that of the previous case by considering 
$T=\{t_3\in N^3(x)\mid t\in V^x_{23}, \forall u\in N(x), d(u,t_3)=2\}$. As we did in the previous claim, we define the set of lines 
$$\cL'=\{\overrightarrow{u_tt_3} \mid t_3\in T\setminus \overrightarrow{xu_t}, u_t\in N(x)\}.$$
Following the steps of previous claim it can be proved that $|\cL'|=|T|$ and that  
$(\cL \cup \cL_{R^x})\cap \cL'=\emptyset$. Moreover, $|R^x|\geq |\{t\in V^x_{12}\mid t_3\notin T\}|$, since for each such $t$, there is $u\in N(x)$ with $(u,t_3)\in R^x$. This shows that 
$$|\cL \cup \cL'\cup \cL_{R^x}|\geq |V|-1.$$

When $R^x$ is not empty we can prove that there is a line $\ell$ containing $x$ but not in $\cL^x\cup \cL'$ which proves the result. Otherwise, , we get that $|T|=|V^x_{23}|$ and that $|\cL \cup \cL'|=|V|-1.$ In this case, the line $\overrightarrow{zt_3}$, for $z$ an in-neighbor of $t_3$ in $N^3(x)$, does not contain $x$, implying that it does not belong to $\cL\cup \cL'$, finishing the proof.
\end{proof}

\end{proof}

\section{Conclusion}

From the result in \cite{ACHKS} we know that the class of thin graphs of diameter at most three is finite.  Whether this is also true for digraphs is an open question. The main conclusion of this work is that no thin bridgeless digraph can be bipartite or has directed girth four, unless it is $K_{2,2}, K_{2,3}$ or the directed cycle $\overrightarrow{C_3}$.

\end{document}